\documentclass{amsart}

\usepackage{amsmath, amssymb}
\input xy
\xyoption{all}
\begin{document}

\newtheorem{theorem}{Theorem}[section]
\newtheorem{lemma}[theorem]{Lemma}
\newtheorem{corollary}[theorem]{Corollary}
\newtheorem{fact}[theorem]{Fact}
\newtheorem{proposition}[theorem]{Proposition}
\newtheorem{claim}[theorem]{Claim}
\theoremstyle{definition}
\newtheorem{example}[theorem]{Example}
\newtheorem{remark}[theorem]{Remark}
\newtheorem{definition}[theorem]{Definition}
\newtheorem{question}[theorem]{Question}

\def\aut{\operatorname{Aut}}
\def\id{\operatorname{id}}
\def\PP{\mathbb{P}}
\def\P{{\bf P}}
\def\GG{\mathbb{G}}
\def\cb{\overline{\operatorname{Cb}}}
\def\tp{\operatorname{tp}}
\def\stp{\operatorname{stp}}
\def\acl{\operatorname{acl}}
\def\dcl{\operatorname{dcl}}
\def\eq{\operatorname{eq}}
\def\th{\operatorname{Th}}
\def\locus{\operatorname{loc}}
\def\dom{\operatorname{dom}}
\def\ccm{\operatorname{CCM}}


\def\Ind#1#2{#1\setbox0=\hbox{$#1x$}\kern\wd0\hbox to 0pt{\hss$#1\mid$\hss}
\lower.9\ht0\hbox to 0pt{\hss$#1\smile$\hss}\kern\wd0}
\def\ind{\mathop{\mathpalette\Ind{}}}
\def\Notind#1#2{#1\setbox0=\hbox{$#1x$}\kern\wd0\hbox to 0pt{\mathchardef
\nn=12854\hss$#1\nn$\kern1.4\wd0\hss}\hbox to
0pt{\hss$#1\mid$\hss}\lower.9\ht0 \hbox to
0pt{\hss$#1\smile$\hss}\kern\wd0}
\def\nind{\mathop{\mathpalette\Notind{}}}

\title[A Model-theoretic counterpart to {M}oishezon morphisms]{A Model-theoretic counterpart to\\ {M}oishezon morphisms}
\author{Rahim Moosa}
\thanks{Rahim Moosa was supported by an NSERC Discovery Grant.}
\address{Department of Pure Mathematics\\
University of Waterloo\\
Canada N2L 3G1}
\email{rmoosa@math.uwaterloo.ca}

\date{April 26, 2010}

\subjclass[2000]{Primary 03C45, 03C98. Secondary 32J27}

\begin{abstract}
The notion of being {\em Moishezon} to a set of types, a natural strengthening of internality motivated by complex geometry, is introduced.
Under the hypothesis of Pillay's~\cite{pillay01} canonical base property, and using results of Chatzidakis~\cite{chatzidakis10}, it is shown that if a stationary type of finite $U$-rank at least two is almost internal to a nonmodular minimal type and {\em admits a diagonal section}, then it is Moishezon to the set of nonmodular minimal types.
This result is inspired by Campana's~\cite{campana81} ``first algebraicity criterion'' in complex geometry.
Other related abstractions from complex geometry, including {\em coreductions} and {\em generating fibrations} are also discussed.
\end{abstract}

\maketitle

\section{Introduction}

\noindent
Compact complex manifolds can be viewed as first order structures in the language where there are predicates for all complex-analytic relations.
The corresponding theory admits quantifier elimination and is of finite Morley rank.
Many of the techniques and approaches used by complex geometers in studying the bimeromorphic structure of compact complex manifolds can be viewed as specialisations of the methods of finite rank geometric stability theory.
But this paper has to do with the other direction;
abstracting ideas {\em from} bimeromorphic geometry {\em to} stability theories.
The goal is to develop new model-theoretic tools that may be useful in the study of other contexts, such as differential-algebraic geometry.

The general model-theoretic theme that we are interested in is the following:
Given a finite rank stationary type $p(x)$  in a stable theory $T$, how does $p$ relate to the nonmodular minimal types of $T$?
If the Zilber trichotomy holds in $T$  (as it does in the theory of compact complex manifolds) then a nonmodular minimal type corresponds to algebraic geometry over an algebraically closed field, and so one is trying to quantify how much of an expansion of algebraic geometry the type $p(x)$ involves.
In the different theories the expansion from algebraic geometry is in different directions:
the theory of compact complex manifolds adds meromorphic structure while the theory of differentially closed fields in characteristic zero, or separably closed fields in positive characteristic, adds (Hasse-Schmidt) differential-algebraic structure.

The notions appearing in this paper were originally developed to prove, under suitable hypotheses on the theory, that every finite rank type has a minimum extension that is internal to the set of nonmodular minimal types.
The difference in rank then, between the type and this minimum extension, would be a new way to quantify how far the type is from the ``algebraic" part of that theory.
In complex geometry such minimum extensions, called {\em algebraic coreductions}, have been known to exist since the early nineteen eighties (due to Camapana~\cite{campana81}).
In its first form this paper was devoted to developing the model-theoretic machinery that would allow the complex-geometric constructions to go through in a wide class of stable theories.
As it turned out however, around the same time, Chatzidakis~\cite{chatzidakis10} found a much more direct proof of this theorem using milder (but similar) assumptions on the theory.
Nevertheless, we feel that some of the ideas and arguments that were extracted and abstracted from complex geometry may be of independent interest and may have other applications.
We therefore present them here.

The paper is laid out as follows:
In section~\ref{section-moish} we introduce a strengthening of internality modelled after the complex-geometric notion of a Moishezon morphism.
In section~\ref{section-crit}, under the assumption that the theory admits the canonical base property (a stability-theoretic condition that is true of compact complex manifolds~\cite{pillay01} and differentially closed fields~\cite{pillayziegler03}),  we give a criterion for when a finite rank type is Moishezon to the set of nonmodular minimals.
We rely heavily here on results of Chatzidakis~\cite{chatzidakis10}, without which we would have had to assume an apparently stronger (but {\em a posteriori} equivalent) property introduced by the author and Pillay in~\cite{cbp}.
Much of the work in this section involves pushing the techniques of~\cite{cbp} as far as they will go.
In any case, our criterion can be viewed as a model-theoretic counterpart of a very special case of Campana's ``first algebraicity criterion" from~\cite{campana81}.
Finally, in section~\ref{section-cored}, we discuss the connection to the existence of coreductions, a problem which we also reformulate in terms of the {\em generating families} that appeared in~\cite{cbp}.

\smallskip

We work throughout in $\overline{M}^{\eq}$ where $\overline M$ is a sufficiently saturated model of a complete stable theory $T$.
All parameter sets are assumed to be of cardinality less than the cardinality of $\overline M$, and
all tuples $a, b,\dots$ will be assumed to be possibly infinite tuples of length strictly less than the cardinality of $\overline M$.
We will only be concerned with precision up to interalgebraicity.
In particular, following~\cite{cbp}, by the {\em canonical base} of a stationary type $p(x)\in S(A)$, which we will denote by $\cb(p)$,
we will mean (an enumeration of) the {\em algebraic} closure of $A_0$, where $A_0$ is the smallest definably closed subset of $\dcl(A)$ over which $p$ does not fork and the restriction of $p$ to which is stationary.

\section{Moishezon: between internal and algebraic}
\label{section-moish}
\noindent
Suppose $\bf P$ is a set of partial types and $p(x)\in S(B)$ is a stationary type.
Here, and throughout this paper, for any set $C\supseteq B$, $p(x)|C$ denotes the nonforking extension of $p$ to $C$ and ${\bf P}_C$ denotes the set of partial types in $\bf P$ whose domains are contained in $C$.
We also use ${\bf P}_C^{\overline M}$ to denote the set of solutions $\displaystyle\bigcup_{q\in{\bf P}_C}q^{\overline M}$.

We will say that $p(x)\in S(B)$ is {\em $\P$-algebraic} if $p^{\overline M}\subseteq \acl\left(B{\bf P}_B^{\overline M}\right)$.
Recall that a stationary type $p(x)\in S(B)$ is {\em almost $\bf P$-internal}
if there is
some $C\supseteq B$ such that $p(x)|C$ is $\P$-algebraic.
We wish to introduce a condition that lies strictly in between $\P$-algebraicity and almost $\P$-internality.
To motivate our definition we turn to the model theory of compact complex manifolds.

Let $\mathcal A$ be the many-sorted first order structure where there is a sort for every reduced and irreducible complex-analytic space and where the basic predicates are the complex-analytic subsets of cartesian of sorts.
The theory of this structure, denoted by $\operatorname{CCM}$, admits quantifier elimimination and is, sort by sort, of finite Morely rank.
A distinguished sort in $\mathcal A$ is the projective line $\mathbb P$.
The complex field is definable on this sort, and in fact the induced structure on $\mathbb P$ is bi-interpretable with the pure field structure on the complex numbers.

Working in a saturated elementary extension of $\mathcal A$, let us consider what it means for a stationary type $p(x)=\stp(a/b)$ to be $\mathbb P$-algebraic or almost $\mathbb P$-internal.
(Here $a$ and $b$ are assumed to be finite tuples.)
If we let $X=\locus(a,b)$ and $Y=\locus(b)$ -- where here $\locus$ denotes the {\em locus}, that is, the smallest complex-analytic set containing the tuple in question -- then, as is explained in section~2 of~\cite{ret}, $p$ is the generic type of a generic fibre of the fibration $X\to Y$ induced by the projection map.
It follows from results in that paper (see also the discussion in section~3 of~\cite{cbp}) that
\begin{itemize}
\item
$p$ is $\mathbb P$-algebraic if and only if $X$ meromorphically embeds into $Y\times\mathbb P_n$ over $Y$, for some $n\geq 0$, and
\item
$p$ is almost $\mathbb P$-internal if and only if $p$ is $\mathbb P$-internal if and only if there is a complex-analytic space over $Y$,  $Y'\to Y$, such that the fibred product $X\times_YY'$ meromorphically embeds into $Y'\times\mathbb P_n$ over $Y'$, for some $n\geq 0$.
\end{itemize}
But in complex geometry there is an intermediate notion that plays a much more important role than either of the above two conditions:
a surjective morphism $X\to Y$ is {\em Moishezon} if $X$ meromorphically embeds into $\mathbb P(\mathcal F)$ over $Y$, where $\mathcal F$ is a coherent analytic sheaf on $Y$ and $\mathbb P(\mathcal F)$ denotes the projective linear space associated to $\mathcal F$.
If $p$ is $\mathbb P$-algebraic then $X\to Y$ is seen to be Moishezon by taking $\mathcal F$ to be the free sheaf $\mathcal O_Y^n$.
The converse fails because not every projective linear space splits.
On the other hand, as every projective linear space does split after base change, Moishezon-ness of $X\to Y$ implies that $p$ is almost $\mathbb P$-internal.
The converse of this implication also fails: the results of~\cite{ret} show that if $\dim(a/b)=1$ then $p$ is $\mathbb P$-internal, but there exist non-Moishezon fibrations of dimension one, for example of Hopf surfaces.
Hence Moishezon-ness of $X\to Y$ lies strictly in between $\mathbb P$-algebraicity and almost $\mathbb P$-internality of $p$.

Since we do not as yet see how to formulate the abstract analogue of a projective linear space, we focus instead on the following properties of Moishezon fibrations:
\begin{itemize}
\item[(1)]
The restriction of a Moishezon morphism to a complex-analytic subspace is again Moishezon.
\item[(2)]
The composition of Moishezon morphisms is Moishezon.
\end{itemize}
The following definition in the general stable setting extracts, almost literally, the model-theoretic content of these properties.

\begin{definition}
\label{moishezon}
Suppose $p(x)\in S(Ab)$ is a stationary type and $\P$ is an $\aut_A(\overline M)$-invariant set of partial types.
We will say that $p(x)$ is {\em $\bf P$-Moishezon over $A$} if whenever $a\models p(x)$ and $c$ is such that $\stp(b/Ac)$ is almost $\bf P$-internal, then $\stp(a/Ac)$ is almost $\P$-internal.
\end{definition}

\begin{remark}
\label{moishrem}
\begin{itemize}
\item[(a)]
By automorphisms, to verify $\P$-Moishezon-ness it suffices to prove that for {\em some} $a\models p(x)$, whenever $c$ is such that $\stp(b/Ac)$ is almost $\bf P$-internal, then $\stp(a/Ac)$ is almost $\P$-internal.
\item[(b)]
A variant of $\P$-Moishezon-ness would be to consider types $\tp(a/Ab)$ such that if $\stp(b/Ac)$ is almost $\P$-internal {\em and $a\ind_{Ab}c$}, then $\stp(a/Ac)$ is $\P$-internal.
This weaker condition is closer to what was implicitly considered by the author and Pillay in~\cite{cbp}.
\item[(c)]
Our definition of $\P$-Moishezon-ness should be viewed as somewhat tentative because it does not seem to precisely generalise the complex-geometric notion of Moishezon morphism.
In Example~\ref{ccm} below we show that the generic type of a Moishezon morphism is $\mathbb P$-Moishezon.
But not all $\mathbb P$-Moishezon types arise in this way:
Suppose $Y$ is a non-algebraic strongly minimal compact complex manifold $f:X\to Y$ is a fibration with $\dim(X)=\dim(Y)+1$.
Then $f$ need not be Moishezon, but the generic type of $f$ will be $\mathbb P$-Moishezon.
Indeed, if $a$ is generic in $X$ then $\tp\big(a/f(a)\big)$ is $\mathbb P$-internal (by one-dimensionality, see~\cite{ret}), and as the only $\mathbb P$-internal extensions of $\tp\big(f(a)\big)$ are algebraic extensions, $\tp\big(a/f(a)\big)$ is in fact be $\mathbb P$-Moishezon.
\end{itemize}
\end{remark}

\begin{example}
\label{ccm}
{\em Working in $\operatorname{CCM}$, suppose $X$ and $Y$ are irreducible compact complex spaces, and $f:X\to Y$ is a Moishezon morphism.
If $a\in X$ is generic, then $\stp\big(a/f(a)\big)$ is $\mathbb P$-Moishezon over $\emptyset$.}
\end{example}

\begin{proof}
Let $b:=f(a)$ and suppose $\stp(b/c)$ is $\mathbb P$-internal.
(Recall that in $\operatorname{CCM}$, almost internality and internality coincide.)
Let $U:=\locus(a,b,c)$, $V:=\locus(b,c)$, and $W:=\locus(c)$.
Then, as $\stp(bc/c)$ is $\mathbb P$-internal, there exists $W'\to W$ such that $V\times_WW'\to W'$ meromorphically embeds into $W'\times \mathbb P_n$ over $W'$.
In particular, $V\times_WW'\to W'$ is Moishezon.

On the other hand, since $f$ is Moishezon, so is $\Gamma(f)\to Y$, where $\Gamma(f)=\locus(a,b)$ is the graph of $f$.
Moishezon morphisms are preserved by base change.
Hence $\Gamma(f)\times_YV\to V$ is Moishezon.
Note that $Y=\locus(b)$.
As $U$ is contained in $\Gamma(f)\times_YV$ over $V$, $U\to V$ is also Moishezon.
Hence $U\times_V(V\times_WW')=U\times_WW'\to V\times_WW'$ is Moishezon.

The composition of Moishezon morphisms is Moishezon.
So $U\times_WW'\to W'$ is Moishezon.
Hence there exists $W''\to W'$ such that $U\times_WW''\to W''$ meromorphically embeds in $W''\times\mathbb P_m$ over $W''$ for some $m$.
It follows that the generic type of $U\to W$, which is $\stp(ab/c)$, is $\mathbb P$-internal.
Hence, $\stp(a/c)$ is $\mathbb P$-internal, as desired.
\end{proof}

The following properties of $\P$-Moishezon-ness are immediate consequences of the of the definitions.

\begin{proposition}
\label{firstprop}
Suppose $p(x)\in S(Ab)$ is stationary and $\P$ is an $\aut_A(\overline M)$-invariant set of partial types.
\begin{itemize}
\item[(a)]
If $p$ is $\P_A$-algebraic then it is $\P$-Moishezon over $A$.
In particular, if $p$ is almost $\bf P_A$-internal then some nonforking extension of $p$ is $\bf P$-Moishezon over $A$.\footnote{Note that, unlike almost $\bf P$-internality, having a $\bf P$-Moishezon nonforking extension does not imply being $\bf P$-Moishezon.}
\item[(b)]
If $p$ is $\bf P$-Moishezon over $A$ then it is almost $\bf P$-internal.
\item[(c)]
If $p\upharpoonright A$ is stationary and almost $\bf P$-internal then $p$ is $\bf P$-Moishezon over $A$.
\item[(d)]
If $p$ is $\bf P$-Moishezon over $A$ then every extension of $p$ is also.
\item[(e)]
If $p$ is $\bf P$-Moishezon over $A$ and $a'\in\acl(Aba)$ for some $a\models p(x)$,
then $\stp(a'/Ab)$ is $\bf P$-Moishezon over $A$.
\item[(f)]
If $\stp(a/Ab)$ and $\stp(a'/Ab)$ are $\bf P$-Moishezon over $A$, then so is $\stp(aa'/Ab)$.
\item[(g)]
If $\stp(a/Ab)$ and $\stp(b/Ac)$ are $\bf P$-Moishezon over $A$,
then so is $\stp(a/Ac)$.
\item[(h)]
If $\stp(a/Abc)$ is $\bf P$-Moishezon over $A$ then it is $\bf P$-Moishezon over $Ab$.
\end{itemize}
\end{proposition}

\begin{proof}
For part~(a) suppose $\stp(b/Ac)$ is almost $\P$-internal.
So there exists $C\supseteq Ac$ such that $b\ind_{Ac}C$ and $b\in \acl\big(C{\bf P}_C^{\overline M}\big)$.
Let $a\models p(x)|Cb$.
Then $a\ind_{Ac}C$ and, since $p^{\overline{M}}\subseteq\acl\big(Ab{\bf P}_A^{\overline M}\big)$, we have $a\in\acl\big(C{\bf P}_C^{\overline M}\big)$.
So $\stp(a/Ac)$ is almost $\bf P$-internal.
Hence $p(x)=\tp(a/Ab)$ is $\P$-Moishezon over $A$.

In particular, if $p$ is almost $\P_A$-internal then from some $b'\supseteq b$, the nonforking extension $q(x):=p(x)|Ab'$ is $\P_A$-algebraic, and hence $\P$-Moishezon over $A$.

For part~(b) just take $c=b$ in the definition of $\P$-Moishezon.

Part~(c).
Suppose $p\upharpoonright A$ is stationary and almost $\bf P$-internal.
Then every extension of $p\upharpoonright A$ is almost $\P$-internal.
In particular $\stp(a/Ac)$ is almost $\bf P$-internal for any $a\models p(x)$ and any $c$ such that $\stp(b/Ac)$ is almost $\bf P$-internal.
That is, $p(x)=\tp(a/Ab)$ is $\P$-Moishezon over $A$.

Part~(d).
Suppose  $q\in S(Ab')$ is an extension of $p$, where $b'\supseteq b$.
Suppose $\stp(b'/Ac)$ is $\bf P$-internal and let $a\models q(x)$.
Then $a\models p(x)$ and $\stp(b/Ac)$ is $\bf P$-internal, so that $\stp(a/Ac)$ is almost $\P$-internal, as desired.

Part~(e).
If $\stp(b/Ac)$ is almost $\bf P$-internal then by the $\P$-Moishezon-ness of $\stp(a/Ab)$ over $A$, $\stp(a/Ac)$ is also almost $\bf P$-internal.
So $\stp(ab/Ac)$ is almost $\bf P$-internal.
Since $a'\in\acl(Aba)$, $\stp(a'/Ac)$ is almost $\bf P$-internal, as desired.

Part~(f).
If $\stp(b/Ac)$ is almost $\bf P$-internal then by $\bf P$-Moishezon-ness both $\stp(a/Ac)$ and $\stp(a'/Ac)$ are almost $\bf P$-internal.
Hence, $\stp(aa'/Ac)$ is almost $\bf P$-internal, as desired.

Part~(g).
If $\stp(c/Ad)$ is almost $\bf P$-internal then by the $\P$-Moishezon-ness of $\stp(b/Ac)$ over $A$, $\stp(b/Ad)$ is also almost $\bf P$-internal.
But then, by the $\P$-Moishezon-ness of $\stp(a/Ab)$ over $A$, $\stp(a/Ad)$ is almost $\bf P$-internal, as desired.

Finally, to prove part~(h), suppose $\stp(c/Abd)$ is almost $\P$-internal.
Then so is $\stp(bc/Abd)$, and so, by the $\P$-Moishezon-ness of $\stp(a/Abc)$, $\stp(a/Abd)$ is almost $\P$-internal, as desired.
\end{proof}

\section{A criterion for being Moishezon to the set of\\ nonmodular minimal types}
\label{section-crit}

\noindent
In this section we take $\bf P$ to be the set of all nonmodular minimal types.
An essential part of understanding the structure of finite-rank definable sets in a stable theory is the task of determining whether a given type bears any relation to the set of nomodular minimal types.
For example, is it non-orthogonal to $\P$? If so, is it almost internal to $\P$?
In this section we are interested in the question: When is an almost $\P$-internal type actually $\P$-Moishezon?
The following algebraicity criterion in complex geometry will serve as a template for us:

\begin{theorem}[Campana~\cite{campana81}, Th\'eore\`me 2]
\label{1algcrit}
Suppose $X$ is a compact complex space of K\"ahler-type, and $f:X\to Y$ is a surjective morphism whose fibres are Moishezon.\footnote{Recall that a compact complex space $A$ is {\em Moishezon} if the morphism $f:A\to\{\text{pt.}\}$ is Moishezon. Equivalently, $A$ is bimeromorphic to a projective algebraic variety.}
If there exists an analytic subspace $A\subseteq X$ such that $f|_A:A\to Y$ is Moishezon, then $f$ is Moishezon.
In particular, if $f$ has a meromorphic section then it is Moishezon.
\end{theorem}

Actually, we have stated the theorem here in a more general way than it appears in~\cite{campana81} where Campana assumes in addition that $Y$ is Moishezon and then concludes that $X$ is Moishezon.
But the proof he gives there goes through word for word in the relative setting.
The model-theoretic significance of the K\"ahler assumption has to do with saturation in an appropriate language, for details about which we refer the reader to~\cite{sat}.
Let us only mention that if $X$ is K\"ahler then generic type of $f$ is $\mathbb P$-internal if and only if the general fibres of $f$ are Moishezon.
So the above theorem is about passing from $\mathbb P$-internality to Moishezon-ness.
To find a model-theoretic counterpart to this theorem (or more accurately, to the ``in particular'' clause) we begin by approximating the notion of a section as follows.

\begin{definition}
\label{mt-section}
Let us say that a complete stationary type $p$ {\em admits a diagonal section} if there exists $a\in\acl\big(\dom(p)\big)$ such that $a\models p\upharpoonright_{\cb(p)}$.
\end{definition}

\begin{remark}
\label{diagsect}
Every stationary type has a nonforking extension that admits a diagonal section: if $a\models p(x)\in S(B)$ then the nonforking extension $p|Ba$ admits a diagonal section, namely $a$ itself.
This is analogous to the fact that for every morphism $X\to Y$, the diagonal map is a section to $X\times_Y X\to X$.
\end{remark}

What one might hope for, at least in the finite $U$-rank setting, is that if $p$ is $\P$-internal and admits a diagonal section then it is $\P$-Moishezon (over the empty set).
We will in fact prove something like this in Theorem~\ref{moishcrit} below, but only under the additional hypothesis of the ``canonical base property,'', which we now discuss in the following digression.

\subsection{The canonical base property}
Motivated by complex-geometric results due to Campana and Fujiki from the nineteen eighties, Pillay~\cite{pillay01} introduced the following condition:
$T$ has the {\em canonical base property} (CBP) if whenever $U(a/A)<\omega$ and $b=\cb(a/Ab)$ then $\stp(b/Aa)$ is almost ${\bf P}$-internal.\footnote{Here $\P$ is still the set of all nonmodular minimal types, but in fact, by the results of~\cite{chatzidakis10}, it changes nothing if we replace $\P$ by the set of {\em all} minimal types (see Remark 1.1(b) of~\cite{cbp}).}
The complex-geometric results alluded to above more or less directly imply that $\operatorname{CCM}$ has the CBP.
In~\cite{pillayziegler03} Pillay and Ziegler show that differentially closed fields have the CBP.
Indeed, no stable theories have been shown to fail the CBP.
Again motivated by the complex-geometric situation, the author and Pillay introduced an apparent strengthening of this condition in~\cite{cbp}, which we called the {\em uniform canonical base property} (UCBP), and which essentially says that if $U(a/A)<\omega$ and $b=\cb(a/Ab)$ then $\stp(b/Aa)$ is almost ${\bf P}$-Moishezon over $A$.
``Essentially'' because we were working with a slightly weaker notion than the Moishezon-ness, see Remark~\ref{moishrem}(b).
In any case, we pointed out that $\operatorname{CCM}$ has the UCBP, and we asked whether or not the CBP always implies the UCBP.
Recently, Chatzidakis has shown that this is indeed the case.
In terms of Moishezon-ness, her theorem can be stated as follows:

\begin{theorem}[Chatzidakis~\cite{chatzidakis10}, Theorem~2.9]
\label{cbp-ucbp}
Suppose $T$ has the CBP.
If $U(a/A)<\omega$ and $b=\cb(a/Ab)$, then $\stp(b/Aa)$ is ${\bf P}$-Moishezon over $A$.
\end{theorem}

One of the ingredients that goes into proving Theorem~\ref{cbp-ucbp} is the following consequence of the CBP, also due to Chatzidakis~\cite{chatzidakis10}: if $U(a/A)<\omega$ and $b=\cb(a/Ab)$, then $\tp\big(b/\acl(Aa)\cap\acl(Ab)\big)$ is almost $\bf P$-internal.
In~\cite{cbp} the author and Pillay gave a more geometrically motivated proof of this fact (using the UCBP).
Given Theorem~\ref{cbp-ucbp}, one is tempted to ask whether or not $\tp\big(b/\acl(Aa)\cap\acl(Ab)\big)$ is in fact $\bf P$-Moishezon over $A$.
While this is too much to ask, our arguments in~\cite{cbp} can be stretched to give:

\begin{lemma}
\label{strong44}
Suppose the CBP holds for $T$, $U(a/A)<\omega$, $b=\cb(a/Ab)$ and $a=\cb(b/Aa)$.
Then $\tp\big(b/\acl(Aa)\cap\acl(Ab)\big) \ | \ Aa$ is $\bf P$-Moishezon over $A$.
\end{lemma}

\begin{proof}
While this is stronger than the statement of Proposition~4.4 in~\cite{cbp} (which only concludes that $\tp\big(b/\acl(Aa)\cap\acl(Ab)\big)$ is almost $\P$-internal), it actually follows from the proof given there.
Let $d$ be such that $Ad=\acl(Aa)\cap\acl(Ab)$.

By Theorem~\ref{cbp-ucbp}, both $\stp(a/Ab)$ and $\stp(b/Aa)$ are $\P$-Moishezon over $A$.
Now, recursively define $a=a_0,a_1,\dots$ and $b=b_0,b_1,\dots$ such that
\begin{itemize}
\item[(i)]
$a_{i+1}\models\stp(a_i/Ab_i)|Ab_ia$
\item[(ii)]
$b_{i+1}\models\stp(b_i/Aa_{i+1})|Aa_{i+1}b$
\item[(iii)]
$\stp(b_i/Aa)$, $\stp(a_i/Aa)$ are $\P$-Moishezon over $A$.
\end{itemize}
This is done as follows.
Suppose we have already defined $a=a_0,a_1,\dots,a_i$ and $b=b_0,b_1,\dots,b_i$.
Note that by the induction hypothesis $a_ib_i\models\stp(ab/A)$.
Hence $\stp(a_i/Ab_i)$ is $\P$-Moishezon over $A$.
Let $a_{i+1}\models\stp(a_i/Ab_i)|Ab_ia$.
Then $\stp(a_{i+1}/Ab_i)$ is $\P$-Moishezon over $A$.
By the inductive hypothesis, $\stp(b_i/Aa)$ is $\P$-Moishezon over $A$.
So, by Proposition~\ref{firstprop}(h), $\stp(a_{i+1}/Aa)$ is $\P$-Moishezon over $A$.
Next let $b_{i+1}\models\stp(b_i/Aa_{i+1})|Aa_{i+1}b$.
As $\stp(b/Aa)$ is $\P$-Moishezon over $A$, so is $\stp(b_{i+1}/Aa_{i+1})$.
We have just shown that $\stp(a_{i+1}/Aa)$ is $\P$-Moishezon over $A$, and so we get, by Proposition~\ref{firstprop}(h) again, that $\stp(b_{i+1}/Aa)$ is $\P$-Moishezon over $A$, as desired.

Having constructed these sequences, note that by~(i) and~(ii), $\acl(Aa_i)\cap\acl(Ab_i)=\acl(Aa)\cap\acl(Ab)=Ad$ and $a_ib_i\models\stp(ab/Ad)$, for all $i\geq 0$.
Moreover, because of finite $U$-rank, eventually,
$$a\ind_{Ad}a_\ell b_\ell$$
for sufficiently large $\ell\geq 0$.
Indeed, this is Lemma~2.2 of~\cite{cbp} -- see in particular the statement marked by $(*)$ in the proof of that Lemma.
So $\tp(b_\ell/Aa)$ is the nonforking extension of $\tp(b/Ad)$, and it is $\P$-Moishezon over $A$ by~(iii).
\end{proof}

We are interested in the following consequence for minimal canonical types.

\begin{corollary}
\label{strong44rank1}
Suppose the CBP holds for $T$ and $U(a/A)<\omega$.
If $U(a/Ab)=1$, $b=\cb(a/Ab)$, and $b\notin\acl(Aa)$, then $\tp\big(a/\acl(Aa)\cap\acl(Ab)\big) \ | \ Aa$ is $\bf P$-Moishezon over $A$.
\end{corollary}

\begin{proof}
Again, let $d$ be such that $Ad=\acl(Aa)\cap\acl(Ab)$.

We first show that $\tp(b/Ad)|Aa$ is $\bf P$-Moishezon over $A$.
Let $a'=\cb(b/Aa)$.
So $a'\in\acl(Aa)$.
If $a'\in\acl(Ab)$ then $\cb(b/Aa)\subseteq Ad$ and hence $b\ind_{Ad}a$.
It follows that $\tp(b/Ad)|Aa=\tp(b/Aa)$.
But by Theorem~\ref{cbp-ucbp}, $\tp(b/Aa)$ is $\P$-Moishezon over $A$ and we are done.
Hence we may assume that $a'\notin\acl(Ab)$.
But then, as $U(a/Ab)=1$, $a\in\acl(Aba')$.
Since $b\ind_{Aa'}a$, this implies that $a\in\acl(Aa')$.
Hence $a$ and $a'$ are interalgebraic over $A$.
So $a=\cb(b/Aa)$ also.
Now Lemma~\ref{strong44} applies and we have that $\tp(b/Ad)|Aa$ is $\bf P$-Moishezon over $A$.

To prove the corollary, we let $a_1b_1\models\tp(ab/Ad)|Aa$ and we show that $\tp(a_1/Aa)$ is $\P$-Moishezon over $A$.
Note that $b\notin\acl(Aa)\implies b_1\notin\acl(Aa_1)\implies b_1\notin\acl(Aa_1a)$, where the last implication is because $b_1\ind_{Aa_1}a$.
Hence there exists $b_2\models\tp(b_1/Aa_1a)$ with $b_2\notin\acl(Ab_1)$.
Let $p_{b_1}$ and $p_{b_2}$ be the conjugates of $\tp(a/Ab)$ obtained by replacing $b$ by $b_1$ and $b_2$ respectively.
As $b_1=\cb(p_{b_1})$, the fact that $b_2\notin\acl(Ab_1)$ means that $p_{b_1}\cup p_{b_2}$ must be a forking extension of $p_{b_1}$.
The latter being of rank $1$ implies that $p_{b_1}\cup p_{b_2}$ is algebraic.
Since $b_1$ and $b_2$ have the same type over $Aa_1$ and $a_1\models p_{b_1}$, we get that $a_1\models p_{b_2}$ also and so $a_1\in\acl(Ab_1b_2)$.
But $\tp(b_1/Aa)=\tp(b_2/Aa)$ is the nonforking extension of $\tp(b/Ad)$ to $Aa$, and is thus $\P$-Moishezon over $A$ by the preceeding paragraph.
Hence, by Proposition~\ref{firstprop} parts~(e) and~(f), $\tp(a_1/Aa)$ is $\P$-Moishezon over $A$.
\end{proof}

We now return to the promised criterion for when an almost internal type is in fact Moishezon.
For the sake of convenience we restrict ourselves to finite rank theories -- though everything can be formulated so as to apply to finite rank types in arbitrary stable theories (with the CBP).

\begin{theorem}
\label{moishcrit}
Suppose $T$ is a finite $U$-rank theory with the CBP, and $\P$ is the set of all nonmodular minimal types.
Suppose $p$ is a stationary type of $U$-rank at least two that is almost internal to a nonmodular minimal type.
If $p$ admits a diagonal section, then $p$ is $\P$-Moishezon over $\emptyset$.

In particular, every stationary type of rank at least two that is almost internal to a nonmodular minimal type becomes $\P$-Moishezon after taking a nonforking extension to a realisation.
\end{theorem}

\begin{proof}
We begin by proving the ``in particular" clause as a separate lemma.

\begin{lemma}
\label{strongpillaymoosa}
Suppose $T$ is a finite $U$-rank theory with the CBP.
Suppose $c=\cb(a/Ac)$ and $\stp(a/Ac)$ is of $U$-rank at least two and almost internal to a nonmodular minimal type.
Then $\stp(a/Ac)|Aca$ is $\P$-Moishezon over $A$.
\end{lemma}

\begin{proof}[Proof of Lemma~\ref{strongpillaymoosa}]
Replacing $a$ by $(a,c)$ we may also assume that $a\in\acl(Ac)$.
Since $\stp(a/Ac)$ is almost internal to a nonmodular minimal type, there exists $c'\supseteq c$ such that
$\stp(a/Ac')$ is a stationary canonical type of $U$-rank $1$ and $\acl(Aca)\cap\acl(Ac')=\acl(Ac)$.
This is Lemma~5.4 of~\cite{cbp} and expresses the existence of a ``rich family of curves'' on any $U$-rank at least two type that is almost internal to a nonmodular minimal type.
Note that $c'\notin\acl(Aa)$:
indeed, if it were, then we would have $c'\in\acl(Ac)$, contradicting the fact that $U(a/Ac)>1$ while $U(a/Ac')=1$.
Now, by Corollary~\ref{strong44rank1} applied to $\stp(a/Ac')$,
 we have that $\stp\big(a/\acl(Aa)\cap\acl(Ac')\big)|Aa$ is $\P$-Moishezon over $A$.
But $\acl(Aa)\cap\acl(Ac')=\acl(Aca)\cap\acl(Ac')=\acl(Ac)$.
Hence $\stp(a/Ac)|Aa$ is $\P$-Moishezon over $A$, as desired.
\end{proof}

Now suppose $p$ is a stationary type almost internal to a nonmodular minimal type and admitting a diagonal section.
Let $c=\cb(p)$.
Recall that $p$ admitting a diagonal section means that there exists $a\in \acl\big(\dom(p)\big)$ such that $\tp(a/c)=p\upharpoonright_c$.
Note that $p$, or rather its unique extension to $\acl\big(\dom(p)\big)$, is an extension of $\tp(a/c)|ca$.
Hence it suffices to show that $\tp(a/c)|ca$ is $\P$-Moishezon over $\emptyset$.
But since $p$ is a nonforking extension of $\tp(a/c)$, the latter is also almost internal to a nonmodular minimal type (the same one).
So we can apply Lemma~\ref{strongpillaymoosa} with $A=\emptyset$ to conclude that $\tp(a/c)|ca$ is $\P$-Moishezon over $\emptyset$, as desired.
\end{proof}

Note that Theorem~\ref{moishcrit} only abstracts a special case of the ``in particular'' clause of Campana's algebraicity criterion (Theorem~\ref{1algcrit}), and that it is not even clear how to formulate a more complete model-theortic counterpart of that result.
It may be that the Zariski-type structures of Hrushovski-Zilber, where it is possible to talk about {\em specialisations}, offer the right context in which to work this out.

\section{Coreductions}
\label{section-cored}

\noindent
As was mentioned in the introduction, the original intention for developing a notion of Moishezon-ness in  the abstract stable setting was to prove that under certain assumptions on the theory (namely UCBP),
for all $a$ and $A$ there exists a minimum algebraically closed set $A\subseteq B\subseteq\acl(Aa)$ such that $\stp(a/B)$ is almost internal to the set of nomodular minimal types.
As it turned out, Chatzidakis~\cite{chatzidakis10}, employing even weaker assumptions (simply CBP), found a much more direct proof of this intended application.\footnote{She then used this result to prove Theorem~\ref{cbp-ucbp} above; that the CBP implies the UCBP.}
In this final section we discuss Chatzidakis' result from the point of view of known results in complex geometry.
Even though it was not in the end necessary, we also explain how Moishezon-ness is related to these issues.

Given an irreducible compact complex space $X$, one of the first constructions in bimeromorphic geometry is to consoider its {\em algebraic reduction}, a holomorphic surjection $f:X\to V$ where $V$ is a projective algebraic variety and such that $f$ induces an isomorphism between the rational function field of $V$ and the meromorphic function field of $X$.
The algebraic reduction map satisfies a certain universal property expressing that $V$ is the maximum algebraic image of $X$.
The model-theoretic content of this is expressed in the following easy fact about finite rank types in arbitrary stable theories:
Suppose $\P$ is any $\aut_A(\overline M)$-invariant set of partial types and $U(a/A)<\omega$. Then there exists $b\in\acl(Aa)$ such that $\stp(b/A)$ is almost $\P$-internal and whenever $c\in\acl(Aa)$ with $\stp(c/A)$ almost $\P$-internal, then $c\in\acl(Ab)$.
Clearly this $b$ is uniquely determined upto interalgebraicity over $A$, and we could call it the {\em $\P$-reduction of $a$ over $A$}.
Note that finding the $\P$-reduction is the first step in a $\P$-analysis of $\stp(a/A)$.
Up to bimeromorphic equivalence, algebraic reductions in complex geometry are just $\mathbb P$-reductions in the theory $\ccm$.

Now, in~\cite{campana81}, Campana considered a dual notion to algebraic reductions: maximum algebraic fibrations (or quotients).
Somewhat loosely speaking, a meromorphic surjection $f:X\to Y$ is an {\em algebraic coreduction of $X$} if the general fibres of $f$ are Moishezon and if every compact complex-analytic family of Moishezon subspaces of $X$ is finer than the fibration induced by $f$.
Campana proved that algebraic coreductions of K\"ahler-type spaces exist.
Campana's proof used (and it seems inspired) his ``first algebraicity criterion''.
This is what lead us to study Moishezon types and to prove Theorem~\ref{moishcrit} above.
In any case, the model-theoretic content of algebraic coreductions is captured by the following definition:

\begin{definition}
\label{coreddef}
Suppose $\P$ is a set of $\aut_A(\overline M)$-invariant partial types.
A {\em $\P$-coreduction of $a$ over $A$} is a tuple $b\in \acl(Aa)$ such that $\stp(a/Ab)$ is almost $\P$-internal and whenever $c\in\acl(Aa)$ with $\stp(a/Ac)$ almost $\P$-internal, then $b\in\acl(Ac)$.
\end{definition}

But do coreductions exist?
Unlike reductions, they are not an immediate consequence of finite rank.
Nevertheless,
Chatzidakis proves -- this is Theorem~2.8 of~\cite{chatzidakis10} -- that {\em if $T$ has the CBP, $\P$ is an $\aut_A(\overline M)$-invariant set of minimal types, and $U(a/A)<\omega$, then a $\P$-coreduction of $a$ over $A$ exists.}
In particular this leads to an attractive transfer of results from complex geometry to differential-algebraic geometry: finite rank differential-algebraic varieties admit maximum fibrations by subvarieties that are almost internal to the field of constants.
This is a new and potentially useful tool in the study of finite-rank differential algebraic varieties.

In order to prove the existence of coreductions for finite rank types it suffices (and is necessary)  to show that if $b_1,b_2\in \acl(Aa)$ are such that $\stp(a/Ab_1)$ and $\stp(a/Ab_2)$ are almost $\P$-internal, then $\stp\big(a/\acl(Ab_1)\cap\acl(Ab_2)\big)$ is almost $\P$-internal.
In the following proposition we prove a weaker version of this fact, but unconditionally (without the CBP).
It serves as a good illustration of the power of Moishezon-ness over internality.

\begin{proposition}
\label{tocored}
We do not assume the CBP.
Suppose $\P$ is any $\aut_A(\overline M)$-invariant set of types and $U(a/A)<\omega$.
If $b_1,b_2\in \acl(Aa)$ are such that $\stp(a/Ab_i)|Aa$ is $\P$-Moishezon over $\acl(Ab_1)\cap\acl(Ab_2)$ for $i=1,2$, then $\tp\big(a/\acl(Ab_1)\cap\acl(Ab_2)\big)$ is almost $\bf P$-internal.
\end{proposition}

\begin{remark}
Note that in the case when the CBP holds and $\P$ is the set of all nonmodular minimals, the hypothesis that $\tp(a/Ab_i)|Aa$ is $\P$-Moishezon over $\acl(Ab_1)\cap\acl(Ab_2)$ would follow by Theorem~\ref{moishcrit} from $\tp(a/Ab_i)$ being simply almost internal to some nonmodular minimal type.
But this does not give a new proof of the existence of coreductions under CBP because Theorem~\ref{moishcrit} itself relied on the existence of coreductions (to get UCBP from CBP).
\end{remark}

\begin{proof}[Proof of Proposition~\ref{tocored}]
For the sake of notational convenience, let us rename $b:=b_1$ and $c:=b_2$.
Passing to algebraic closures we assume that $\tp(a/Ab)$ and $\tp(a/Ac)$ are stationary.
Replacing $A$ by $\acl(Ab)\cap\acl(Ac)$, we also assume that $A=\acl(Ab)\cap\acl(Ac)$.
We wish to show that $\tp(a/A)$ is almost $\P$-internal.

Recursively define $a=a_0, a_1,\dots$, $b=b_0, b_1,\dots$, and $c=c_0, c_1,\dots$ so that
\begin{itemize}
\item[(i)]
$a_ib_ic_i\models\tp(abc/A)$ for all $i$,
\item[(ii)]
for $i$ even, $a_{i+1}\models\tp(a_i/Ac_i)|Ac_ib$ and $c_{i+1}=c_i$,
\item[(iii)]
for $i$ odd, $a_{i+1}\models\tp(a_i/Ab_i)|Ab_ib$ and $b_{i+1}=b_i$, and
\item[(iv)]
$\stp(a_i/Ab)$ is almost $\P$-internal for all $i$.
\end{itemize}
This is done as follows.
Suppose that for even $i$ we have defined $a=a_0, a_1,\dots, a_i$, $b=b_0, b_1,\dots,b_i$, and $c=c_0, c_1,\dots,c_i$ as desired.
We first show how to define $a_{i+1},b_{i+1},c_{i+1}$.
Let $a_{i+1}\models\tp(a_i/Ac_i)|Aa_ib$.
(Note that $c_i\in\acl(Aa_i)$.)
Since $\tp(a/Ac)|Aa$ is $\P$-Moishezon over $A$ by assumption, we get that $\tp(a_{i+1}/Aa_i)=\tp(a_i/Ac_i)|Aa_i$ is $\P$-Moishezon over $A$.
By~(iv) then, $\stp(a_{i+1}/Ab)$ is almost $\P$-internal.
Set $c_{i+1}:=c_i$.
Since $a_{i+1}c_{i+1}\models\tp(a_ic_i/A)=\tp(ac/A)$, there exists $b_{i+1}$ such that $a_{i+1}b_{i+1}c_{i+1}\models\tp(abc/A)$.

We now construct $a_{i+2}b_{i+2}c_{i+2}$ analogously.
Let $a_{i+2}\models\tp(a_{i+1}/Ab_{i+1})Aa_{i+1}b$.
Since $\tp(a/Ab)|Aa$ is $\P$-Moishezon over $A$,
$\tp(a_{i+2}/Aa_{i+1})=\tp(a_{i+1}/Ab_{i+1})|Aa_{i+1}$ is $\P$-Moishezon over $A$.
But we have just shown that $\stp(a_{i+1}/Ab)$ is almost $\P$-internal.
So $\stp(a_{i+2}/Ab)$ is almost $\P$-internal.
Setting $b_{i+2}:=b_{i+1}$ we can find $c_{i+2}$ such that $a_{i+2}b_{i+2}c_{i+2}\models\tp(abc/A)$.
This completes the recursive construction.

We claim that for all $i$, $U(a_{i+1}/Ab)\geq U(a_i/Ab)$.
Indeed, for $i$ even,
\begin{eqnarray*}
U(a_{i+1}/Ab)
&=&
U(a_{i+1}c_{i+1}/Ab) \ \ \ \text{ as $c_{i+1}\in\acl(Aa_{i+1})$}\\
&=&
U(a_{i+1}c_i/Ab) \ \ \ \text{ as $c_{i+1}=c_i$}\\
&=&
U(a_{i+1}/Abc_i)+U(c_i/Ab)\\
&=&
U(a_{i+1}/Ac_i)+U(c_i/Ab) \ \ \ \text{ as $a_{i+1}\ind_{Ac_i}b$}\\
&=&
U(a_i/Ac_i)+U(c_i/Ab) \ \ \ \text{ as $a_{i+1}\models\tp(a_i/Ac_i)$}\\
&\geq&
U(a_i/Abc_i)+U(c_i/Ab)\\
&=&
U(a_ic_i/Ab)\\
&=&
U(a_i/Ab)
\end{eqnarray*}
The argument for $i$ odd is analogous.

Hence, eventually for some $\ell>0$, $U(a_j/Ab)=U(a_\ell/Ab)$ for all $j\geq \ell$.
Let us take $\ell$ to be even.
Now,
\begin{eqnarray*}
U(a_{\ell+1}b_{\ell+1}/Ab)
&=&
U(a_{\ell+1}/Ab) \ \ \ \text{ as $b_{\ell+1}\in\acl(Aa_{\ell+1})$}\\
&=&
U(a_{\ell+2}/Ab) \ \ \ \text{ by choice of $\ell$}\\
&=&
U(a_{\ell+2}b_{\ell+1}/Ab) \ \ \ \text{ as $b_{\ell+1}=b_{\ell+2}\in\acl(Aa_{\ell+2})$}\\
&=&
U(a_{\ell+2}/Abb_{\ell+1})+U(b_{\ell+1}/Ab)\\
&=&
U(a_{\ell+2}/Ab_{\ell+1})+U(b_{\ell+1}/Ab)\ \ \ \text{ as $a_{\ell+2}\ind_{Ab_{\ell+1}}b$}\\
&=&
U(a_{\ell+1}/Ab_{\ell+1})+U(b_{\ell+1}/Ab)\ \ \ \text{ as $a_{\ell+2}\models\tp(a_{\ell+1}/Ab_{\ell+1})$.}
\end{eqnarray*}
On the other hand,
$U(a_{\ell+1}b_{\ell+1}/Ab)=U(a_{\ell+1}/Abb_{\ell+1})+U(b_{\ell+1}/Ab)$.
It follows that $U(a_{\ell+1}/Abb_{\ell+1})=U(a_{\ell+1}/Ab_{\ell+1})$, which means that $\displaystyle b\ind_{Ab_{\ell+1}}a_{\ell+1}$.
But $\displaystyle b\ind_{Ac_{\ell+1}}a_{\ell+1}$ by~(ii).
Hence, $\cb(b/Aa_{\ell+1})\subseteq\acl(Ab_{\ell+1})\cap\acl(Ac_{\ell+1})=A$ and we have that $\displaystyle b\ind_Aa_{\ell+1}$.
So $\tp(a/A)=\tp(a_{\ell+1}/A)$ has a nonforking extension -- namely $\stp(a_{\ell+1}/Ab)$-- that is almost $\P$-internal.
Hence $\tp(a/A)$ is itself almost $\P$-internal.
\end{proof}

\subsection{Generating fibrations}
In conclusion, let us relate some of these ideas to the notion of ``generating families'' from~\cite{cbp}.
By a {\em fibration} of a stationary type $p(x)=\tp(a/A)$ we mean simply a stationary type $q(x,y)=\tp(a,b/A)$ such that $s(y)=\tp(b/A)$ and $q_b:=\tp(a/Ab)$ are also stationary.
The {\em fibres} of this fibration are the types $q_{b'}:=q(x,b')$ where $b'\models s(y)$.
Given another realisation $a'\models p(x)$, we say that $a$ and $a'$ are {\em $q$-connected} if there exists a finite sequence $a=a_0,a_1,\dots, a_\ell=a'$ of realisations of $p(x)$ and a sequence $b=b_0,b_1,\dots,b_{\ell-1}$ of realisations of $s(y)$ such that $a_i$ and $a_{i+1}$ both realise $q_{b_i}(x)$ for all $i=0,\dots,\ell-1$.
We say that {\em $q(x,y)$ generates $p(x)$} if every (equivalently some) $A$-independent pair of realisations of $p(x)$ are $q$-connected.
It turns out that in the finite $U$-rank context $q$ generates $p$ if and only if $\acl(Aa)\cap\acl(Ab)=\acl(A)$ (see Lemma~2.2 of~\cite{cbp}).
The author and Pillay obtained the following criterion for internality to a nomodular minimal type (Theorem~1.3(b) of~\cite{cbp}):
If $T$ has the CBP and a finite rank type $p(x)$ is generated by a fibration whose fibres are almost internal to a nonmodular minimal type, then $p(x)$ is itself almost internal to that type.

We can extend these notions to pairs (or indeed finite collections) of fibrations.
Let us say that {\em $p(x)$ is generated by the pair of fibrations $q_1(x,y)=(a,b/A)$ and $q_2(x,z)=\tp(a,c/A)$} if for some (equivalently any) pair of $A$-independent realisations of $p(x)$, $a$ and $a'$, there exists a finite sequence of realisations of $p(x)$, $a=a_0,a_1,\dots,a_\ell=a'$ such that each pair $a_i,a_{i+1}$ is either $q_1$-connected or $q_2$-connected.
Chatzidakis' proof of the existence of coreductions more or less gives the following internality criterion:

\begin{theorem}
Suppose $T$ has the CBP, $p(x)\in S(A)$ is a stationary finite rank type, and $\P$ is any set of $\aut_A(\overline M)$-invariant minimal types.
If $p$ is generated by a pair of fibration whose fibres are almost $\P$-internal, then $p$ is itself almost $\P$-internal.
\end{theorem}

\begin{proof}[Sketch of proof]
Suppose $p(x)=\tp(a/A)$ is generated by the fibrations $q_1(x,y)=\tp(a,b/A)$ and $q_2(x,z)=\tp(a,c/A)$.
Since $\tp(a/Aa), \tp(a/Ab)$, and $\tp(a/Ac)$ are all almost $\P$-internal, the existence of $\P$-coreductions (Theorem~2.8 of~\cite{chatzidakis10}) implies that $\tp\big(a/\acl(Aa)\cap\acl(Ab)\cap\acl(Ac)\big)$ is almost $\P$-internal.
But the fact that $q_1$ and $q_2$ generate $p(x)$ implies that
$$\acl(Aa)\cap\acl(Ab)\cap\acl(Ac)=A.$$
Indeed, arguing as in Lemma~2.2 of~\cite{cbp}, it is not hard to show that there exists a finite sequence $abc=a_0b_0c_0,a_1b_1c_1,\dots,a_nb_nc_n$ such that $\stp(a_ib_ic_i/A)=\stp(a_{i+1}b_{i+1}c_{i+1}/A)$ and either $b_i=b_{i+1}$ or $c_i=c_{i+1}$, for all $i=0,\dots,n-1$, and $a_n$ is independent of $a$ over $A$.
Hence $\acl(Aa)\cap\acl(Ab)\cap\acl(Ac)=\acl(Aa_n)\cap\acl(Ab_n)\cap\acl(Ac_n)$, and so $\acl(Aa)\cap\acl(Ab)\cap\acl(Ac)\subseteq \acl(Aa)\cap\acl(Aa_n)=A$.
Thus, $p(x)=\tp(a/A)$ is almost $\P$-internal.
\end{proof}

Proposition~\ref{tocored} can also be understood in these terms.
Loosely speaking, it says that if a finite rank stationary type over $A$ is generated by a pair of fibrations whose fibres satisfy a certain strengthening of almost $\P$-internality, where $\P$ here is any $\aut_A(\overline M)$-invariant set of partial types, then the type itself is almost $\P$-internal.
The precise strengthening of internality that we require of the fibres is that a nonforking extension to a realisation is $\P$-Moishezon over $A$.


\end{document}